\documentclass[a4paper,12pt]{article}
\usepackage{latexsym,amssymb,amsmath,amsthm,amsxtra,xypic}
\usepackage{amsfonts}
\usepackage{amscd}
\usepackage{mathrsfs}
\usepackage[dvips]{graphicx}
\usepackage[all]{xy}

\newtheorem{thm}{Theorem}[section]
\newtheorem{lemma}[thm]{Lemma}
\newtheorem{prop}[thm]{Proposition}

\theoremstyle{definition}
\newtheorem{defn}[thm]{Definition}

\newtheorem{remark}[thm]{Remark}
\numberwithin{equation}{section}

\def\trr{\triangleright}
\def\trl{\triangleleft}

\def\ppl{\leftharpoonup}
\def\ppr{\rightharpoonup}

\begin{document}
\allowdisplaybreaks

\title{Unified products for alternative and pre-alternative algebras}

\author{Tao Zhang, Shuxian Cui and Jing Si}


\date{}
\maketitle

\footnotetext{{\it{Keyword}: Alternative algebra, extending structure, non-abelian cohomology, matched pair, complement}}

\footnotetext{{\it{Mathematics Subject Classification (2010)}}: 17A30, 18G60.}

\begin{abstract}
The theory of unified product and extending structures  for alternative and pre-alternative algebras are developed.
 It is proved that the extending structures of these algebras can be classified by using some non-abelian cohomology and deformation map theory.
\end{abstract}

\maketitle

\section{Introduction}

As a generalization of associative algebras, alternative algebra has been studied  from different aspects by mathematicians.
The first well-known non-associative example of alternative algebra is Cayley's octonion numbers.
The structure theory of the finite-dimensional alternative algebras was studied by M. Zorn in \cite{Zorn}.
The theory of representations of alternative algebras was given by R. D. Schafer in \cite{Sch1} and N. Jacobson in \cite{Jac},
Alternative bialgebra and pre-alternative bialgebra are investigated in \cite{Gon} and \cite{Bai2}.
For more general alternative algebra theory, see \cite{Sch,ZSSS,EM}.

The extending structures problems for groups, associative algebras, Hopf algebras, Lie algebras, Leibniz algebras, left-symmetric algebras and Lie conformal algebras
have been studied in \cite{AM1, AM2, AM3,AM4, AM5, Hong1,Hong2,Hong3} respectively.

%

 In this paper,  we study extending structures for alternative  and pre-alternative algebras.
This paper is organized as follows. In Section 2, some preliminaries about alternative algebras are recalled. In Section 3, we introduce the concept of unified product $A\natural V$ of alternative algebras associated with an extending datum $\Omega(A, V)$. The sufficient and necessary condition to ensure that $A\natural V$ with a given canonical product is an alternative algebra is given.
Then, we show that there exists an alternative algebra structure on $E$ such that $A$ is a subalgebra of $E$ if and only if $E$ is isomorphic to a unified product of $A$ and $V$.
 In Section 4, some special cases of unified products are given.
  In Section 5,  we study flag extending structures of unified products.
 In Section 6, the classifying complements problem for alternative algebras is studied.

Throughout this paper, all vector spaces are assumed to be over an algebraically closed field $K$ of characteristic not equal to 2 and 3.
Let $V$ be a vector space. The identity map from $V$ to $V$ is denoted by $id_V$ or $id$.

\section{Preliminaries}
In this section, we will recall some basic definitions and facts about alternative algebras and  pre-alternative algebras.

\begin{defn}
An alternative algebra is a vector space $A$ with a multiplication $\circ: A\times A\to A: (x,y)\mapsto x\circ y$
such that the following identities hold:
\begin{eqnarray}
(x, y, z)=-(y, x, z),\quad (x, y, z)=-(x, z, y)
\end{eqnarray}
where $(x, y, z)=(x \circ y) \circ z - x \circ(y \circ z)$ is the associator of the elements $x, y, z\in A$.
\end{defn}

Note that the obove identity is equivalent to the following identity:
\begin{eqnarray}
&&(x \circ y) \circ z - x\circ(y \circ z) + (y\circ x)\circ z - y\circ (x \circ z) = 0,\\
&&(x \circ y) \circ z - x\circ(y \circ z) + (x\circ z)\circ y - x\circ (z \circ y) = 0.
\end{eqnarray}

\begin{defn}
Let $A$ be an alternative algebra, $V$  be a vector space. A bimodule of $A$ over the vector space $V$ is a pair of linear maps $ \trr :A\times V \to V,(x,v) \to x \trr v$ and $\trl :V\times A \to V,(v,x) \to v \trl x$  such that the following conditions hold:
\begin{eqnarray}
(x \circ y) \trr v - x \trr (y\trr v) + (y \circ x) \trr v - y\trr (x \trr v) = 0,\\
(v \trl x) \trl y - v \trl (x \circ y) + (v \trl y) \trl x - v\trl (y \circ x) = 0,\\
v \trl (x \circ y) - (v \trl x) \trl y + x \trr (v \trl y) - (x \trr v) \trl y = 0,\\
(x \circ y) \trr v -x\trr(y\trr v)+(x\trr v)\trl y-x\trr(v \trr y)=0,
\end{eqnarray}
hold for all $x,y\in A$ and $v\in V.$
\end{defn}

\begin{prop}
Let $A$ be an alternative algebra, $V$ be a vector space. Then
$V$  is a $A$-bimodule if and only if $A\oplus V$ is an alternative algebra under the following multiplication:
\begin{eqnarray}
(x, u)\circ(y, v)\triangleq(x\circ y, x\trr v+u\trl y),
\end{eqnarray}
for all $x,y\in A$ and $u, v\in V.$
\end{prop}

\begin{defn}
A pre-alternative algebra is a quadruple $(A, \prec, \succ)$ in which $A$ is a vector space,
 $\prec, \succ : A\otimes A\rightarrow A$ are bilinear maps  such that
for all $x, y, z\in A$,
\begin{eqnarray}
(x\circ y)\succ z - x \succ(y\succ z)+ (y\circ x)\succ z -  y \succ(x\succ z)&=&0, \\
(x\prec y)\prec z - x \prec(y\circ z)+ (x\prec z)\prec y -  x \prec(z\circ y)&=&0,\\
 (x\succ y)\prec z - x \succ(y\prec z)+ (y\prec x)\prec z -  y \prec(x\circ z)&=&0,\\
(x\succ y)\prec z - x \succ(y\prec z)+ (x\circ z)\succ y -  x \succ(z\succ y)&=&0,
\end{eqnarray}
where we denote by $x\circ y= x\succ y+x\prec y$.
\end{defn}

\begin{thm}
 Let $(A, \prec, \succ)$ be a pre-alternative algebra. If we define the operation
$$x\circ y\triangleq x\prec y+x\succ y.$$
Then $(A, \circ)$ is an alternative algebra, which is called the associated alternative algebra of
 $(A, \prec, \succ)$
and denoted by $Alt(A)=(A, \circ)$.
\end{thm}

\begin{defn}
 Let $(A, \prec, \succ)$ be a pre-alternative algebra.
An $A$-bimodule is a vector space $V$ together with four bilinear maps $\prec,\succ:A \times V\to V,\ (x,v)\mapsto x
\prec v,  (v,x)\mapsto v\succ x$, $\prec,\succ:V \times A\to V, (v,x)\mapsto v\prec x, (v,x)\mapsto v\succ x$
satisfying the following conditions:
\begin{eqnarray}
(x\circ y+ y\circ x)\succ v &=& x \succ(y\succ v)+
  y \succ(x\succ v),\\
(x\circ v+ v\circ x)\succ y &=& x \succ(v\succ y)-  v \succ(x\succ y),\\
(v\prec x)\prec y + (x\succ v)\prec y &=& v \prec(x\circ y)+
 x \succ(v\prec y),\\
(x\prec v)\prec y + (v\succ x)\prec y &=& x \prec(v\circ y)+
 v \succ(x\circ y),\\
(y\prec x)\prec v + (x\succ y)\prec v &=& y \prec(x\circ v)+
  x \succ(y\prec v),\\
(y\succ v)\prec x + (y\circ x)\succ v &=& y \succ(v\prec x)
+  y \succ(x\succ v),\\
(v\succ)\prec x + (v\circ x)\succ y &=& v \succ(y\prec x)+
  v \succ(x\succ y),\\
(y\succ x)\prec v + (y\circ v)\succ x &=& y \succ(x\prec v)+
  y \succ(v\succ x),\\
(v\prec x)\prec y + (v\prec y)\prec x &=& v \prec(x\circ y+ y\circ x),
\label{pbm9}\\
(x\prec v)\prec y +(x\prec y)\prec v &=& x \prec(v\circ y+y\circ v),
\label{pbm10}
\end{eqnarray}
where $x\circ y=x\prec y+x\succ y, x\circ v=x\succ v+x\prec v, v\circ x=v\succ x+v\prec x$.
\end{defn}

\begin{defn}
Let $A$ be an alternative algebra, $E$ a vector space such that $A$ is a subspace
of $E$ and $V$ a complement of $A$ in $E$. For a linear map $\varphi: E\rightarrow E$, the following
diagram is considered:
$$\xymatrix{ {A}\ar[d]^{id}\ar[r]^{i}& {E}\ar[d]^{\varphi}\ar[r]^{\pi} & {V}\ar[d]^{id} \\
{A}\ar[r]^{i}&{E}\ar[r]^{\pi} &{V} }$$
where $\pi: E\rightarrow V$ is the natural projection of $E=A\oplus V$ onto $V$ and
$i: A\rightarrow E$ is the inclusion map. We say that
$\varphi: E\rightarrow E$ \emph{stabilizes} $A$ (resp. \emph{co-stabilizes} $V$) if the left square
(resp. the right square) of the  above diagram is commutative.

Let $\circ$ and $\circ'$ be two alternative  structures on $E$ both containing $A$ as an alternative   subalgebra.
If there exists an alternative algebra isomorphism $\varphi: (E, \circ)\rightarrow (E,\circ')$ which stabilizes
$A$,  $\circ$ and $\circ'$ are called \emph{equivalent}, which is denoted by
$(E,\circ)\equiv (E,\circ')$.

If there exists an alternative algebra isomorphism $\varphi: (E, \circ)\rightarrow (E,\circ')$ which stabilizes
$A$ and co-stabilizes $V$,  then $\circ$ and $\circ'$ are called \emph{cohomologous}, which is denoted by
$(E,\circ)\approx (E,\circ')$.
\end{defn}

Obviously, $\equiv $ and $\approx $ are equivalence relations on the set of all alternative algebra structures on
$E$ containing $A$ as an alternative subalgebra. Denote by $\text{Extd}(E,A)$ (resp. $\text{Extd}'(E,A)$) the set
of all equivalence classes via $\equiv $ (resp. $\approx $). Thus, $\text{Extd}(E,A)$ is the classifying object of the extending structures problem and $\text{Extd}'(E,A)$ provides a classification of the extending structures problem from the point of the view of the extension problem. In addition, it is easy to see that there exists a canonical projection $\text{Extd}(E,A)\twoheadrightarrow \text{Extd}'(E,A)$.

\section{Unified products for alternative algebras}
In this section, we will introduce the concept of unified product for alternative algebras and give a theoretical answer to the
extending structures problem.
\begin{defn}
Let $(A,\circ)$ be an alternative algebra  and $V$ a vector space. An \emph{extending datum} of $A$ by $V$ is
a system $\Omega(A, V)$ consisting four linear maps
\begin{eqnarray*}
\trr :A\times V \to V,\quad\trl:V\times A \to V,
\quad \ppr :V\times A \to A, \quad\ppl:A\times V \to A,
\end{eqnarray*}
 and two bilinear maps
 \begin{eqnarray*}
 *:V\times V \to V,\quad \omega :V\times V \to A.
\end{eqnarray*}
Let $\Omega(A, V)=(\trr,\trl,\ppr,\ppl,*,\omega)$ be an extending datum. Denote by $A\natural V$ the direct sum vector space
$A\oplus V$ with multiplication:
\[
(x, u)\bullet (y, v) = \big(x \circ y + u\ppr y + x \ppl v + \omega({u,v} ),\,
u * v + x \trr v + u \trl y\big).
\]
for all $ x,y \in A,u,v \in V$. Then $A\natural V$ is called the \emph{unified product} of $A$ and $V$ if it is an alternative algebra with the multiplication given by above. In this case, the extending datum $\Omega(A, V)$ is called a \emph{extending structure} of $A$ by $V$.
\end{defn}

\begin{thm}
Let $A$ be an alternative algebra, $V$ be a vector space and $\Omega(A,V)$ an extending datum of $A$ by $V$. Then
$A\natural V$ is a unified product if and only if the following conditions hold for all $x,y\in A$, $u,v\in V$:

\noindent(A1)
\begin{eqnarray*}&&
 (u \ppr x) \circ z + (x \ppl u) \circ z + (u \trl x) \ppr z + (x \trr u) \ppr z \\
 &=& u \ppr (x \circ z) + x \circ (u \ppr z) + x\ppl (u \trl z),
\end{eqnarray*}
(A2)
\begin{eqnarray*}
(u \trl x) \trl z + (x\trr u) \trl z= u \trl (x \circ z) + x \trr (u \trl z),
\end{eqnarray*}
(A3)
\begin{eqnarray*}&&
 \omega (u,v) \circ z + \omega (v,u) \circ z + (u * v) \ppr z + (v * u) \ppr z \\
 &=& u \ppr (v \ppr z) + v \ppr (u \ppr z) + \omega(u,v \trl z) + \omega (v,u \trl z),
\end{eqnarray*}
(A4)
\begin{eqnarray*}&&
 (u * v) \trl z + (v * u)\trl z \\
 &=& u * (v \trl z) + v * (u \trl z) + u \trl (v \ppr z) + v \trl (u \ppr z),
\end{eqnarray*}
(A5)
\begin{eqnarray*}&&
 (u \ppr x) \ppl w + (x \ppl u)\ppl w + \omega (u \trl x,w) + \omega (x \trr u,w) \\
 &=& u \ppr (x \ppl w) + x \ppl (u * w) + x \circ (\omega (u,w)) + \omega(u,x \trr w),
\end{eqnarray*}
(A6)
\begin{eqnarray*}&&
 (u \trl x) * w + (x \trr u) * w + (u \ppr x) \trr w + (x\ppl u) \trr w \\
 &=& u * (x \trr w) + u \trl (x\ppl w) + x \trr (u * w),
\end{eqnarray*}
(A7)
\begin{eqnarray*}&&
 (x \circ y) \ppl w + (y \circ x)\ppl w \\
 &=& x \circ (y \ppl w) + y \circ (x \ppl w) + x \ppl (y \trr w) + y \ppl(x \trr w),
\end{eqnarray*}
(A8)
\begin{eqnarray*}
(x \circ y) \trr w + (y \circ x)\trr w = x \trr (y \trr w) + y\trr (x \trr w),
\end{eqnarray*}
(A9)
\begin{eqnarray*}&&
 \omega (u,v) \ppl w +\omega (v,u) \ppl w + \omega (u * v,w)+ \omega (v * u,w) \\
 &=& u \ppr \omega (v,w) + v \ppr \omega (u,w) + \omega (u,v * w) + \omega (v,u * w),
\end{eqnarray*}
(A10)
\begin{eqnarray*}&&
 (u * v) * w + (v * u) * w + \omega (u,v) \trr w + (\omega (v,u)) \trr w \\
 &=& u * (v * w) + v * (u * w) + u \trl(\omega (v,w)) + v \trl (\omega (u,w)),
\end{eqnarray*}
(A11)
\begin{eqnarray*}&&
 (u \ppr y) \circ z + (u \trl y) \ppr z+ (u \ppr z) \circ y + (u \trl z) \ppr y \\
 &=& u \ppr (y \circ z) + u \ppr (z \circ y),
\end{eqnarray*}
(A12)
\begin{eqnarray*}
(u \trl y) \trl z + (u \trl z) \trl y = u \trl (y \circ z) + u\trl (z \circ y),
\end{eqnarray*}
(A13)
\begin{eqnarray*}&&
 (x \ppl v) \circ y + (x \trr v) \ppr y + (x \circ y) \ppl v \\
 &=& x \circ (v \ppr y) + x \ppl (v \trl y) + x \circ (y \ppl v) + x \ppl (y \trr v),
 \end{eqnarray*}
 (A14)
\begin{eqnarray*}
(x \trr v) \trl y + (x \circ y) \trr v = x \trr (v \trl y) + x \trr (y \trr v),
\end{eqnarray*}
 (A15)
\begin{eqnarray*}&&
 (u * v) \ppr y + (u \ppr y) \ppl v +\omega (u \trl y,v) + (\omega (u,v)) \circ y \\
 &=& u \ppr (v \ppr y) + u \ppr (y \ppl v) +\omega (u,y \trr v) + \omega (u,v\trl y),
\end{eqnarray*}
 (A16)
\begin{eqnarray*}&&
 (u * v) \trl y + (u \trl y)* v + (u \ppr y) \trr v \\
 &=& u * (v \trl y) + u * (y \trr v) + u \trl (y \ppl v) + u \trl(v \ppr y),
\end{eqnarray*}
 (A17)
\begin{eqnarray*}&&
 (x \ppl v) \ppl w + (x \ppl w) \ppl v + \omega (x \trr w,v) + \omega(x \trr v,w) \\
 &=& x \circ \omega (w,v) + x \circ \omega (v,w) + x \ppl (w * v)+ x \ppl (v * w),
\end{eqnarray*}
 (A18)
\begin{eqnarray*}&&
 (x \trr v) * w + (x \trr w) * v + (x \ppl w) \trr v + (x\ppl v) \trr w \\
 &=& x \trr (w * v) + x \trr (v * w),
 \end{eqnarray*}
  (A19)
\begin{eqnarray*}&&
 (u * v) * w + (u * w) * v + (\omega(u,v)) \trr w + (\omega (u,w)) \trr v \\
 &=& u * (v * w) + u * (w * v) + u \trl(\omega (v,w)) + u \trl (\omega (w,v)).
\end{eqnarray*}
\end{thm}

\begin{proof}
Define
\begin{eqnarray*}
R((x,u),(y,v),(z,w))=((x,u),(y,v),(z,w))+((y,v),(x,u),(z,w)),
\end{eqnarray*}
where $(x,u)$, $(y,v)$, $(z,w)\in A\oplus V$.
Note that $A\natural V$ is an alternative algebra if and only if $R((x,u),(y,v),(z,w))=0$
for all $x,y,z\in A$ and $u,v,w\in V$. The proof is by direct but tedious computations.
\end{proof}

\begin{remark}
In fact, from conditions (A2), (A8), (A12) and (A14), one show that $V$ is a bimodule of $A$.
\end{remark}

Given an  extending structure $\Omega(A,V)$. It is obvious that
$A$ can be seen an alternative subalgebra of $A\natural V$. Conversely, we will prove that any alternative algebra structure on a vector space $E$ containing $A$ as a subalgebra is isomorphic to a unified product.

\begin{thm}
Let $(A, \circ)$ and $(E, \circ)$ be  alternative algebras  such that $E$ containing $A$ as a subalgebra of $E$. Then, there exists an extending structure $\Omega(A,V)$ of $A$ by a subspace $V$ of $E$ and an isomorphism of alternative algebras $(E,\circ)\cong A\natural V$ which stabilizes $A$ and co-stabilizes $V$.
\end{thm}

\begin{proof}
Note that there is a natural linear map $p: E\rightarrow A$ such that $p(x)=x$ for all $x\in A$.
Set $V=\text{Ker}(p)$ which is a complement of $A$ in $E$. Then, we define the extending datum $\Omega(A,V)$ of $A$ by a subspace $V$ of $E$ as follows:
\begin{eqnarray*}
&& \ppr : V \times A \to A, \quad
u \ppr x := p (u \circ x), \quad  \trl : V \times A \to V,
\quad u \trl x := u\circ x  - p (u\circ x) \\
&& \ppl    : A \times V \to A, \quad
x \ppl u := p(x\circ u),
\quad \trr  : A \times V \to V, \quad  x \trr u := x \circ u - p(x\circ u)\\
&& \omega : V \times V \to A, \quad    \omega(u, v) := p (u \circ v), \quad
\ast : V \times V \to V, \quad u \ast v := u \circ v - p (u \circ v)
\end{eqnarray*}
for all $x,y \in A,u,v \in V$. It is easy to see that $\varphi: A\times V\rightarrow E$ defined as
$\varphi(x,u)=x+u$ is a linear isomorphism, whose inverse is as follows: $\varphi^{-1}(e):=(p(e),e-p(e))$ for all $e\in E$.
Next, we should prove that $\Omega(A,V)$ is a  extending structure of $A$ by $V$ and $\varphi: A \natural V \rightarrow E$ is an isomorphism of alternative algebras that stabilizes $A$ and co-stabilizes $V$. In fact, if $\varphi: A\times V\rightarrow E$ is an isomorphism of alternative algebras, there exists
a unique left-symmetric   product given by
\begin{eqnarray}
(x,u)\circ (y,v)=\varphi^{-1}(\varphi(x,u)\circ \varphi (y,v)).
\end{eqnarray}

Therefore, for completing the proof, we only need to check that the product defined by above is just the one given in  the above extending system $\Omega(A,V)$.
Indeed, we have
\begin{eqnarray*}
(x,u)\circ (y,v)&=&\varphi^{-1}(\varphi(x,u)\circ \varphi(y,v))=\varphi^{-1}((x+u)\circ (y+v))\\
&=&\varphi^{-1}(x\circ y+x\circ v+u\circ y+u\circ v)\\
&=&(x\circ y+p(x\circ v)+p(u\circ y)+p(u\circ v),\\
&&x\circ v+u\circ y+u\circ v-p(x\circ v)-p(u\circ y)-p(u\circ v))\\
&=& (x \circ y + u \ppr y + x \ppl v + \omega(u,v),\, u * v + x \trr v + u \trl y)\\
&=&(x, u)\bullet (y, v).
\end{eqnarray*}
for all $x,y \in A,u,v \in V$.

Therefore, $\varphi: A\natural V\rightarrow E$ is an isomorphism of alternative algebras and
the following diagram is commutative
$$\xymatrix{ {A}\ar[d]^{id}\ar[r]^{i}& {A\natural V}\ar[d]^{\varphi}\ar[r]^{q} & {V}\ar[d]^{id} \\
{A}\ar[r]^{i}&{E}\ar[r]^{\pi} &{V} }$$
where $q:A\natural V\rightarrow V$ and $\pi: E\rightarrow V$ are the natural projections.
The proof is finished.
\end{proof}

Next, by the above Theorem, for classifying all alternative algebra structures on $E$ containing $A$ as a subalgebra,
we only need to classify all unified products $A\natural V$ associated to all alternative algebra structures
$\Omega(A,V) $ for a given complement $V$ of $A$ in $E$.

\begin{lemma}
Let $\Omega(A,V)$ and $\Omega'(A,V)$ be two  extending structures of $A$ by $V$ and $A\natural V$, $A\natural' V$ be the corresponding
unified products. Then, there is a bijection between the set of all morphisms of alternative algebras
$\varphi: A\natural V\rightarrow A\natural' V$ which stabilizes $A$ and the set of pairs $(r,{s})$, where
${r}: V\rightarrow A$, ${s}: V\rightarrow V$ are linear maps satisfying the following conditions:
\begin{eqnarray}
{s}(u \trl x) &=& {s}(u)\trl ' x,\\
{s}(x \trr u) &=& x \trr ' {s}(u),\\
 {s}(u \ast v) &=& r(u) \trr ' {s}(v) + {s}(u)\trl' r(v) + {s}(u) \circ ' {s}(v),\\
 r(u \trl x) &=& r(u)x - u \ppr x +s(u) \ppr ' x,\\
 r(x \ppl u) &=& x\circ r(u) - x\ppl u + x \ppl ' v(u),\\
 \notag r(u \ast v) &=& r(u)\circ r(v) + \omega' ({s}(u),\, {s}(v)) \\
 &&- \omega(u,\, v) + r(u) \ppl ' {s}(v) + {s}(u) \ppr ' r(v),
\end{eqnarray}
for all $x,y \in A,u,v \in V$.

The bijection from the set of all morphisms of alternative algebras
$\varphi_{{r},{s}}: A\natural V\rightarrow A\natural' V$ to the set of pairs $(r,{s})$ is given as follows:
\begin{eqnarray*}
\varphi(x,u)=(x+{r}(u),{s}(u)).
\end{eqnarray*}
In addition, $\varphi_{{r},{s}}$ is an isomorphism if and only if ${s}: V\rightarrow V$ is an isomorphism, and
$\varphi_{{r},{s}}$ costabilizes $V$ if and only if ${s}= id_V$.
\end{lemma}

\begin{proof}
Let $\varphi: A\natural V\rightarrow A\natural' V$ be a homomorphism of alternative algebras.
Since $\varphi$ stabilizes $A$, $\varphi(x,0)=(x,0)$. Moreover, we can set $\varphi(0,u)=({r}(u),{s}(u))$, where
${r}: V\rightarrow A$, ${s}: V\rightarrow V$ are two linear maps. Therefore,
we get $\varphi(x,u)=(x+{r}(u),{s}(u))$. Then, we should prove that
$\varphi$ is a homomorphism of alternative algebras if and only if the above conditions hold.
It is enough to check that
\begin{eqnarray}
\varphi((x,u)\bullet (y,v))=\varphi(x,u)\bullet' \varphi(y,v)
\end{eqnarray}
holds for all generators of $A\natural V$.
Obviously,  this equation holds for the pair $(x,0)$, $(y,0)$.
Then, we consider  the pair $(x,0)$, $(0,u)$.
According to \begin{eqnarray*}
\varphi((x,0)\bullet (0,u))&=&\varphi(x\ppl u, x\trr u)\\
&=&(x\ppl u+{r}(x\ppl u),{s}(x\trr u),
\end{eqnarray*}
and
\begin{eqnarray*}
\varphi(x,0)\bullet \varphi(0,u)&=&(x,0)\circ ({r}(u),{s}(u))\\
&=&(x\circ {r}(u)+x\ppl {s}(u), x\trr{s}(u)),
\end{eqnarray*}
we get that this equation holds for the pair $(x,0)$, $(0,u)$ if and only if the first two conditions hold.
Similarly, it is easy to check that the other conditions hold for the pair $(0,u)$, $(x,0)$  and $(0,u)$, $(0,y)$.

Assume that $\varphi_{{r},{s}}$ is bijective. It is obvious that ${s}$ is surjective. Then, we only need to prove that
${s}$ is injective. Let $x\in V$ such that ${s}(u)=0$. Then, we get
$\varphi_{{r},{s}}(-{r}(u),x)=(-{r}(u)+{r}(u),{s}(u))=(0,0)$. Thus, $x=0$, i.e. ${s}$ is injective.
Conversely, assume that ${s}:V\rightarrow V$ is bijective. Then, $\varphi_{{r},{s}}$ has the inverse given by
$\varphi_{{r},{s}}^{-1}(x,u)=(x-{r}({s}^{-1}(u)),{s}^{-1}(u))$. Thus, by the first part, $\varphi_{{r},{s}}$
is an isomorphism. Therefore, $\varphi_{{r},{s}}$ is an isomorphism if and only if ${s}: V\rightarrow V$ is an isomorphism.
Finally, it is obvious that $\varphi_{{r},{s}}$ costabilizes $V$ if and only if ${s}= id_V$.
The proof is finished.
\end{proof}

\begin{defn}
Two $\Omega(A,V)$ and  $\Omega'(A,V)$ are called \emph{equivalent} if there exists a pair of linear maps $(r,{s})$ where ${r}: V\rightarrow A$, ${s}: V\to V$ such that
the following conditions hold:
\begin{gather}
 u \trl x ={s}^{-1}({s}(u)\trl' x),\\
x \trr u ={s}^{-1}( x \trr' {s}(u)),\\
 u \circ v = {s}^{-1}(r(u) \trr ' {s}(v) + {s}(u)\trl' r(v) + {s}(u) \circ ' {s}(v)),\\
 u \trl x = r^{-1}(r(u)x - u \ppr x + v(u) \ppr ' x),\\
 x \ppl u = r^{-1}(xr(u) - x\ppl u + x \ppl ' s(u)),\\
u \circ v = r^{-1}(r(u)r(v) + \omega ' ({s}(u),\, {s}(v)) - \omega(u,\, v) + r(u) \ppl ' {s}(v) + {s}(u) \trr ' r(v)),
\end{gather}
for all $x,y \in A,u,v \in V$,
then $\Omega(A,V)$ and  $\Omega'(A,V)$ are called \emph{equivalent} and we denote it
by $\Omega(A,V)\equiv \Omega'(A,V)$.

Moreover, in case $s=id_V$, the above conditions are reduced to
\begin{eqnarray}
u \ppr x &=& r(u)\cdot x + u \ppr ' x - r(u \trl ' x), \\
x \ppl u &=& x\cdot r(u) + x \ppl ' u - r(x \trr ' u),\\
u \ast v &=& r(u) \trr ' v + u \trl ' r(v) + u \ast' v,\\
\omega(u,\, v) &=& r(u)\cdot r(v) + r(u) \ppl ' v + u
\ppr ' r(v) + \omega' (u, \, v)\\ &&-r \bigl(r(u)
\trr ' v + u \trl ' r(v) + u \ast' v\bigl).
\end{eqnarray}
$\Omega(A,V)$ and  $\Omega'(A,V)$ are called \emph{cohomologous} and we denote it
by $\Omega(A,V)\approx \Omega'(A,V)$.
\end{defn}

Then, by the above discussion, the answer for the extending structures problem of alternative algebras  is given as follows:
\begin{thm}
Let $A$ be an alternative algebra, $E$ a vector space that contains $A$ as a subspace and $V$ a complement
of $A$ in $E$. Then, we get:\\
(1)Denote $\mathcal{H}_A^2(V,A):=\mathfrak{T}(A,V)/\equiv$. Then, the map
\begin{eqnarray}
\mathcal{H}_A^2(V,A)\rightarrow Ext(E,A),~~~~\overline{\Omega(A,V)}\rightarrow (A\natural V,\circ)
\end{eqnarray}
is bijective, where $\overline{\Omega(A,V)}$ is the equivalence class of $\Omega(A,V)$ under $\equiv$.\\
(2) Denote $\mathcal{H}^2(V,A):=\mathfrak{T}(A,V)/\approx$. Then, the map
\begin{eqnarray}
\mathcal{H}^2(V,A)\rightarrow Ext'(E,A),~~~~\overline{\overline{\Omega(A,V)}}\rightarrow (A\natural V,\circ)
\end{eqnarray}
is bijective, where $\overline{\overline{\Omega(A,V)}}$ is the equivalence class of $\Omega(A,V)$ under $\approx$.\\
\end{thm}

Finally, we give the definition of unified product for pre-alternative algebras.
\begin{defn}
Let $(A, \prec , \succ)$ be a pre-alternative algebra  and $V$ be a vector space. An \emph{extending datum} of $A$ by $V$ is
a system consisting eight linear maps
\begin{eqnarray*}
\prec , \succ :A\times V \to V,\quad\prec , \succ :V\times A \to V,\\
 < , >  :V\times A \to A,\quad < , >  :A\times V
\to A
\end{eqnarray*}
and four bilinear maps
\begin{eqnarray*}
< , >  :V\times V \to V, \quad\omega_\prec ,\omega_\succ: A\times A \to V.
\end{eqnarray*}
Let $\Omega(A, V)$ be an extending datum. Denote by $A\natural V$ the direct sum vector space
$A\oplus V$ with the products
\[
(x, u) \ll (y, v) = \big(x \prec y +  x < v + u< y + \omega_\prec (u,v),\, u < v +x \prec v + u \prec y\big),
\]
\[
(x, u) \gg (y, v) = \big(x \succ y + x > v + u> y + \omega_\succ (u,v),\, u > v +x \succ v + u \succ y\big).
\]
for all $x,y \in A,u,v \in V$. Then $A\natural V$ is called the \emph{unified product} of $A$ and $V$ if it is a pre-alternative algebra with the products given by above. In this case, the extending datum $\Omega(A, V)$ is called a \emph{extending structure} of $A$ by $V$.
\end{defn}

The conditions for a unified product to be a pre-alternative algebra are in the Appendix of this paper.

\section{Special cases of unified products}

In this section, we show that crossed products and matched pairs of two alternative algebras are both special cases of unified
products.

\subsection{Crossed products}

\begin{defn}
Let $(A, \circ)$ and $(B, *)$ be two alternative algebras.
Then $(A,B)$  is called a crossed system if there exists  bilinear maps
$$\ppr:B\times A \to A,\quad \ppl:A\times B\to A, \quad \omega: B\times B\to A$$
 such that the following products on the direct sum space $A\oplus B$:
\begin{eqnarray}
(x, u)\bullet (y, v) = (x \circ y + u \ppr y + x \ppl v+\omega(u,v),\, u * v)
\end{eqnarray}
define an alternative algebra structure. This alternative algebra is called the crossed product of $A$ and $B$ and we denote this alternative algebra by $A \#_\omega B$.
\end{defn}

\begin{thm}
Let $A$ and $B$ be two alternative algebras. Then $(A,B)$  is  a crossed products  if and only if the following conditions hold:
\begin{eqnarray}
 (u \ppr x) \circ z + (x \ppl u) \circ z= u \ppr (x \circ z) + x \circ (u \ppr z)
\end{eqnarray}
\begin{eqnarray}
 (x \circ y) \ppl w + (y \circ x)\ppl w= x \circ (y \ppl w) + y \circ (x \ppl w),
\end{eqnarray}
\begin{eqnarray}
 (u \ppr y) \circ z + (u \ppr z) \circ y= u \ppr (y \circ z) + u \ppr (z \circ y),
\end{eqnarray}
\begin{eqnarray}
 (x \ppl v) \circ y  + (x \circ y) \ppl v= x \circ (v \ppr y)  + x \circ (y \ppl v),
 \end{eqnarray}
 \begin{eqnarray}
\notag&& (u \ppr x) \ppl w + (x \ppl u)\ppl w\\
 &=& u \ppr (x \ppl w) + x \ppl (u * w) + x \circ \omega (u,w),
\end{eqnarray}
\begin{eqnarray}
\notag&&u \ppr (v \ppr y) + u \ppr (y \ppl v)  \\
 &=&  (u * v) \ppr y + (u \ppr y) \ppl v + \omega (u,v) \circ y ,
\end{eqnarray}
 \begin{eqnarray}
\notag&& (x \ppl v) \ppl w + (x \ppl w) \ppl v  \\
 &=&x \circ \omega (w,v) + x \circ \omega(v,w)+ x \ppl (w * v)+ x \ppl (v * w),
\end{eqnarray}
\begin{eqnarray}
\notag&& u \ppr (v \ppr z) + v \ppr (u \ppr z) \\
 &=& \omega (u,v) \circ z + \omega (v,u) \circ z + (u * v) \ppr z + (v * u) \ppr z,
\end{eqnarray}
\begin{eqnarray}
\notag&& \omega (u,v) \ppl w + (\omega (v,u)) \ppl w + \omega (u * v,w) + \omega (v * u,w)\\
 &=& u \ppr \omega (v,w) + v \ppr \omega (u,w) + \omega (u,v * w) + \omega (v,u * w),
\end{eqnarray}
\end{thm}

\subsection{Matched pair and the factorization problem}

\begin{defn}
Let $(A, \circ)$ and $(B, *)$ be two alternative algebras.
Then $(A,B)$  is called a matched pair if there exists bilinear maps
$$\trr:A\times B \to B,\quad \trl:B\times A \to B,\quad\ppr:B\times A \to A,\quad \ppl:A\times B\to A$$
 such that the following multiplication on the direct sum space $A\oplus B$:
\begin{eqnarray}
(x, u)\bullet (y, v) = (x \circ y + u \ppr y + x \ppl v,\, u * v + x
\trr v + u \trl y)
\end{eqnarray}
define an alternative algebra structure.
The above multiplication is called bicrossed product of $A$ and $B$.
We will denote it by $A\bowtie B$.
\end{defn}

\begin{thm}
Let $A$ and $B$ be two alternative algebras.
Then $(A,B)$  is  a matched pair if and only if the following conditions hold:
\begin{eqnarray}
\notag&&\left( {x \trr u + u \trl x} \right)\ppr y +
(u \ppr x + x \ppl u) \circ y\\
&&\qquad= u \ppr (x \circ y) + x \ppl (u
\trl y) + x \circ (y \ppl u),\qquad\\
\notag&&
(x \circ y + y \circ x) \ppl u\\
&&\qquad= x \ppl (y \trr u) + x
\circ (y \ppl u) + y \ppl (x \trr u) + y \circ (x
\ppl u),\qquad\\
\notag&&
(x \circ y) \ppl u + (x \trr u) \ppr y + (x \ppl u)
\circ y \\
&&\qquad= x \ppl (u \trl y + y \trr u) + x \circ (u
\ppr y + y \ppl u),\qquad\\
\notag&&u \ppr (x \circ y + y \circ x) \\
&&\qquad= (u \ppr x) \circ y + (u \trl x) \ppr
y + (u \ppr y) \circ x + (u \trl y) \ppr x,\qquad\\
\notag&&(u \ppr x + x \ppl u) \trr v + (x \trr u + u
\trl x) * v \\
&&\qquad= x \trr (u * v) + u \trl (x
\ppl v) + u * (x \trr v),\qquad\\
\notag&&(u * v + v * u) \trl x \\
&&\qquad= u \trl (v \ppr x) + u * (v
\trl x) + v \trl (u \ppr x) + v * (u \trl x),\qquad\\
\notag&&(u * v) \trl x + (u \ppr x) \trr v + (u \trl x) *
v \\
&&\qquad= x \trl (v \ppl x + x \ppr v) + u * (v \trr x + x
\trl v),\qquad\\
\notag&& x \trr (u * v + v * u) \\
&&\qquad= (x \trr u) * v + (x \ppl
u) \trr v + (x \trr v) * u + (x \ppl v)
\trr u
\end{eqnarray}
for all $x,y\in A$ and $u,v\in V.$
\end{thm}

\begin{defn}
Let $(A, \prec _A , \succ _A )$ and $(B, < , >)$ be two pre-alternative algebras.
Then $(A,B)$  is called a matched pair if there exits eight bilinear maps
$$ \prec , \succ :A\times B \to B,\quad
\prec , \succ :B\times A \to B,\quad < , >  :B\times A \to A,\quad < , >  :A\times B
\to A$$
 such that the following products on the direct sum space $A\oplus B$:
\begin{eqnarray}
(x, u) \ll (y, v) = (x \prec y + x < v + u< y,\, u < v + x \prec v + u \prec y),\\
(x, u) \gg (y, v) = (x \succ y + x > v + u> y,\,  u > v + x \succ v + u \succ y)
\end{eqnarray}
define a pre-alternative algebra structure. We denote this pre-alternative algebra by $A\bowtie B$.
\end{defn}

\subsection{Classifying complements for alternative algebras}
In this subsection, we will study the classifying complements problem for alternative algebras using the concept of deformation map.
An alternative subalgebra $B$ of $(E,\circ)$ is called an $A$-complement of $(E,\circ)$) if $E=A\oplus B$.
If $B$ is an $A$-complement in $(E,\circ)$, then we get $E\cong A\bowtie B$ for some bicrossed product of $A$ and $B$.
For an alternative subalgebra $A$ of $(E,\circ)$, denote $\mathcal{F}(A,E)$ the set of the isomorphism classes of all $A$-complements in $E$.

\begin{defn}
Let $(A, B)$ be a matched pair of alternative algebras. A linear map
${r}: B\rightarrow A$ is called a \emph{deformation map} of the matched pair $(A, B)$ if ${r}$
satisfies the following condition for all $u$, $v\in B$:
\begin{gather}
{r}(u\circ v)-{r}(u)\circ {r}(v)=u\ppr {r}(v)+{r}(u)\ppl v-{r}({r}(u)\trr v+u\trl {r}(v)).
\end{gather}
Denote by the set of all deformation maps of the matched pair $(A, B)$ by $\mathcal{DM}(B,A)$.
\end{defn}

\begin{thm}
Let $A$ be an alternative subalgebra of $(E,\circ)$, $B$ a given $A$-complement of $E$ with the associated matched pair $(A,B)$.

(1)
Given a deformation map ${r}: B\rightarrow A$. Let $f_{r}: B\rightarrow E=A\bowtie B$ be the linear map defined as
$$f_{r}(u)=({r}(u), u)$$
for all $u\in B$.
Then $\widetilde{B}:=\text{Im}(f_{r})$ is an alternative subalgebra of $E=A\bowtie B$.

Let ${r}: B\rightarrow A$ be a deformation map of the matched pair. Then,
$B_{{r}}:=B$ as a vector space is an alternative algebra with the new product given as follows: $\forall u$, $v\in B$:
\begin{eqnarray}\label{eq:defomprod}
u\circ_{{r}} v:=u\circ v+{r}(u)\trr v+u\trl {r}(v)) .
\end{eqnarray}
$B_{{r}}$ is called the ${r}$-deformation of $B$. Moreover, $B_{{r}}$ is an $A$-complement of $E$.\\

(2) $\overline{B}$ is an $A$-complement of $E$ if and only if $\overline{B}$ is isomorphic to $B_{r} $ for some deformation map
${r}: B\rightarrow A$ of the matched pair $(A,B)$.
\end{thm}

\begin{proof}
(1) Given a deformation map ${r}: B\rightarrow A$.
By the definition of deformation map, we get that for all $u$, $v\in B$,
\begin{eqnarray*}
&& (r(u), u)\bullet (r(v), v)\\
&{=}& \Bigl(r(u)\circ r(v)+ u \ppr r(v) + r(u) \ppl v, \,
u\circ v  + u \trl r(v) + r(u)\trr v \Bigl)\\
&{=}& \Bigl(r (u\circ v + u \ppr r(v)+ r(u)\trr v), \, u\circ v + u
\trl r(v) + r(u)\trr v\Bigl)
\end{eqnarray*}
Therefore, $[({r}(u), u),({r}(v), v)]\in \text{Im}(f_{r})$. Thus, $\text{Im}(f_{r})$ is an alternative subalgebra of $E=A\bowtie B$.
 It is easy to see that $A\cap \text{Im}(f_{r})=\{0\}$ and $(x,u)=(x-{r}(u),0)+({r}(u), u)\in A+B$ for all $x\in A$, $u\in B$. Hence,
$\text{Im}(f_{r})$ is an $A$-complement of $E=A\bowtie B$.

Next we prove that $B_{r} $ and $\text{Im}(f_{r})$ are isomorphic as alternative algebras. Denote by $\widetilde{f_{r}}:
B\rightarrow \text{Im}(f_{r})$ the linear map induced by $f_{r}$. Obviously, $\widetilde{f_{r}}$ is a linear isomorphism.
Now we prove that $\widetilde{f_{r}}$ is an alternative algebra homomorphism if the product of $B$ is given by the equation in \eqref{eq:defomprod}.
For any $u$, $v\in B$, we get
\begin{eqnarray*}
\widetilde{f_{r}}\bigl(u\circ_r v\bigl)&{=}&\widetilde{f_{r}}\bigl(u\circ v + u \trl r(v) + r(u)\trr v\bigl)\\
&{=}& \Bigl({r \bigl(u\circ v + u \trl r(v) +r(u)\trr v\bigl)},\, u\circ v +
u\trl r(v)+ r(u)\trr v\Bigl)\\
&{=}& \Bigl(r(u)\circ r(v) + u \ppr r(v) + r((u) \ppl v,\, u\circ v + u \ppr r(v)+ r(u)\trr v\Bigl)\\
&{=}&(r(u),\, u)\bullet(r(v), \, v)
= \widetilde{f_{r}}(u)\bullet\widetilde{f_{r}}(v)
\end{eqnarray*}
Thus, $B_{r}$ is an alternative algebra.

(2)  The proof is similar as in \cite[Theorem 5.3]{AM3} so we omit the details.
\end{proof}

\begin{defn}
Let $(A, B)$ be a matched pair of alternative algebras. For two deformation maps
${r}$, $r': B\rightarrow A$, if there exists $\sigma: B\rightarrow B$ a linear automorphism of $B$ such that
for all $u$, $v\in B$:
\begin{eqnarray*}
&&\sigma \bigl(u\cdot v\bigl) -\sigma(u)\cdot \sigma(v)\\
&=& \sigma(u) \trl r' \bigl(\sigma(v)\bigl) + r'\bigl(\sigma(u)\bigl) \ppr \sigma(v) - \sigma\bigl(u\trl r(v)\bigl) - \sigma \bigl(r(u) \trr v\bigl).
\end{eqnarray*}
Then ${r}$ and $r'$ are called \emph{equivalent}. Denote it by ${r}\sim r'$.
\end{defn}

\begin{thm}
Let $A$ be an alternative subalgebra of $E$, $B$ an $A$-complement of $E$ and
$(A, B)$ the associated matched pair. Then, $\sim $ is an equivalence relation on
the set $\mathcal{DM}(B,A)$ and the map
\begin{eqnarray*}
\mathcal{HC}^2(B,A):=\mathcal{DM}(B,A)/\sim \rightarrow \mathcal{F}(A,B), ~~~\overline{{r}}\mapsto B_{r},
\end{eqnarray*}
is a bijection between $\mathcal{HC}^2(B,A)$ and the isomorphism classes of all $A$-complements of $E$.
In particular,  we define the factorization index of $A$ in $E$ as $[E:A]:=|\mathcal{F}(A,E)|=|\mathcal{HC}^2(B,A)|$.
\end{thm}
\begin{proof}
It is easy to see that two deformation maps ${r}$ and $r'$ are equivalent if and only if the corresponding
alternative algebras $B_{r}$ and $B_{r'}$ are isomorphic. Thus we obtain the result.
\end{proof}

\section{Flag extending structures}

In this section, we mainly study the extending structure of $A$ by a 1-dimensional vector space $V$ which is called flag extending structures.

\begin{defn}
Let $A$ be an alternative algebra. A \emph{flag datum} of $A$ consists of the following datum: $\lambda,\mu :A \to k$ are algebraic maps and $D,T:A \to A$ are linear maps satisfying for $x,y\in A$, $x_0\in A$ and $k_0\in K$:

\noindent
(C1)
\begin{eqnarray*}
\mu ( {D( x )} ) = 0,\quad\lambda ( {T( x )} ) = 0,
\end{eqnarray*}
(C2)
\begin{eqnarray*}
T( {x_0 } ) = D( {x_0 } ),\quad\lambda ( {x_0 } ) = \mu ( {x_0 } ),
\end{eqnarray*}
(C3)
\begin{eqnarray*}
\mu ( {x \circ y} ) = \mu ( x )\mu ( y ),\quad
\lambda ( {x \circ y} ) = \lambda ( x )\lambda (y ),
\end{eqnarray*}
(C4)
\begin{eqnarray*}
D^2( x ) + \mu ( x )x_0 = x_0 \circ x + k_0 D( x),
\end{eqnarray*}
(C5)
\begin{eqnarray*}
T^2( x ) + \lambda ( x )x_0 = x \circ x_0 + k_0T( x )_0 ,
\end{eqnarray*}
(C6)
\begin{eqnarray*}
\mu ( x )k_0 + \lambda ( {D( x )} ) =\lambda ( x )k_0 + \mu ( {T( x )} ) + \nu( {D( x )} ),
\end{eqnarray*}
(C7)
\begin{eqnarray*}
\lambda ( x )k_0 + \mu ( {T( x )} ) = \nu( x )k_0 + \lambda ( {D( x )} ) + \lambda( {T( x )} ),
\end{eqnarray*}
(C8)
\begin{eqnarray*}
\notag&& x \circ D( y ) + \mu ( y )T( x ) + D({x \circ y} ) \\
 &=&  D( x ) \circ y + T( x ) \circ y + \mu ( x)D( y ) + \lambda ( x )D( y ),
\end{eqnarray*}
(C9)
\begin{eqnarray*}
\notag&& T( x ) \circ y + \lambda ( x )D( y ) +T( {x \circ y} ) \\
 &=& x \circ D( y ) + \mu ( y )T( x ) + x\circ T( y ) + \lambda ( y )T( x ),
\end{eqnarray*}
(C10)
\begin{eqnarray*}
T( {x \circ y} + {y \circ x} ) =  x \circ T( y ) + y \circ T( x ) + \lambda ( y)T( x ) + \lambda ( x )T( y ),
\end{eqnarray*}
(C11)
\begin{eqnarray*}
D( {x \circ y} + {y \circ x} ) = D( x ) \circ y + D( y ) \circ x + \mu ( x
)D( y ) + \mu ( y )D( x ),
\end{eqnarray*}
(C12)
\begin{eqnarray*}
T^2( x ) + T( {D( x )} ) + \mu ( x)x_0 = D( {T( x )} ) + k_0 T( x ) + x\circ x_0 ,
\end{eqnarray*}
(C13)
\begin{eqnarray*}
D^2( x ) + D( {T( x )} ) + \lambda ( x)x_0 = k_0 D( x ) + T( {D( x )} ) +x_0 \circ x.
\end{eqnarray*}
Denote by $\mathcal{F}(A)$ the set of all flag datums of $A$.
\end{defn}

\begin{prop}
Let $A$ be an alternative algebra and $V$ a vector space of dimension $1$ with
basis $\{u\}$. Then there exists a bijection between the set
$Exd(A, V)$ of all  extending structures of $A$ through $V$ and the set ${\mathcal F} \, (A)$ of
all flag datums of $A$.

Through the above bijection, the unified product corresponding to
$(\Lambda, \, \lambda, \, D, \, d, \, a_0, \, u) \in {\mathcal
F}(A)$ will be denoted by $A \ltimes_{(\lambda, \, \mu, \, D,
\, T, \, x_0,)} \, \{u\}$ and has the multiplication given for
any $x$, $y \in A$ by:
\begin{equation}
(x, u) \bullet (y, u) := \Big( x\circ y + T(x) + D(y) + x_0,\,  \lambda(x) \, u+\mu(y) u+  k_0 \, u \Big)
\end{equation}
That is $A \ltimes_{(\Lambda, \, \lambda, \, D, \, d, \, a_0, \, u)}\, x$ is the algebra generated by the algebra $A$ and $x$ subject
to the relations:
\begin{equation}
u\bullet u = x_0 + k_0 \, u, \quad x\bullet u = T(x) + \lambda (x) \, u, \quad u\bullet x =
D(x) + \mu (x) \, u.
\end{equation}
\end{prop}

\begin{defn}
Let $A$ be a pre-alternative algebra. A \emph{flag datum} of $A$ is a 4-tuple:
$$(\lambda_\prec,\lambda_\succ,\mu_\prec ,\mu_\prec,D_ <,D_>,T_ <,T_ >, x_0,y_0,k_0,l_0 )$$
where $\lambda _\prec,\lambda_\succ,\mu_\prec,\mu_\succ :A \to K$ are linear functions, $D_ <,D_ >,T_ <,T_ > :A \to A$ are linear maps and $x_0,y_0 \in A,k_0,l_0 \in K$ satisfying for all  $x,y\in A$:

\noindent
(P1)
\begin{eqnarray*}
T_ > (x_0 ) = D_ > (x_0 ) + D_ > (y_0 ) + l_0 x_0 , ,
\end{eqnarray*}
(P2)
\begin{eqnarray*}
\mu_\succ (x_0 ) = l_0 k_0 + \lambda_\succ (x_0 ) + \lambda_\succ (y_0),
\end{eqnarray*}
(P3)
\begin{eqnarray*}
 D_\diamond (x) \succ y + T_\diamond (x) \succ y + \mu_\circ (x)T_> (y) + \lambda_\circ (x)T_ > (y) \\
 = \mu_\succ (y)T_ > (x) + T_ > (x \succ y) + x \succ T_ > (y).
\end{eqnarray*}
(P4)
\begin{eqnarray*}
\mu_\succ (y)\mu_\circ (x) + \lambda_\prec (x)\mu_\succ (y) = \mu_\succ (x \succ y),
\end{eqnarray*}
(P5)
\begin{eqnarray*}
D_ > (T_\diamond (x)) + D_ > (D_\diamond (x)) + \mu_\circ (x)x_0 +\lambda_\prec (x)x_0 = k_0 D_ > (x) + T_ > (D_ > (x)) + x \succ x_0 ,
\end{eqnarray*}
(P6)
\begin{eqnarray*}
\lambda_\succ (D_\diamond (x)) + \lambda_\succ (T_\diamond (x))+ \lambda_\prec (x)k_0 + \mu_\circ (x)k_0 = \lambda_\succ (x)k_0 + \mu_ \succ (D_ > (x)),
\end{eqnarray*}
(P7)
\begin{eqnarray*}
D_ > (x \circ y) + D_ > (y \circ x) = \lambda { }_ \succ (y)D_ > (x) +\lambda { }_ \succ (x)D_ > (y) + y \succ D_ > (x) + x \succ D_ > (y),
\end{eqnarray*}
(P8)
\begin{eqnarray*}
\lambda_\succ (x \circ y) + \lambda_\succ (y \circ x) = 2\lambda_\succ(y)\lambda_\succ (x),
\end{eqnarray*}
(P9)
\begin{eqnarray*}
(k_0 + l_0 )T_ > (x) + (k_0 + l_0 )x_0 \succ x + (k_0 + l_0 )y_0 \succ x =T_ > (T_ > (x)) + \mu_\succ (x)x_0 ,
\end{eqnarray*}
(P10)
\begin{eqnarray*}
l_0 \mu_\succ (x) = \mu_\succ (T_ > (x)),
\end{eqnarray*}
(P11)
\begin{eqnarray*}
D_ < (y_0 ) = k_0 y_0 + T_ < (x_0 ) + T_ < (y_0 ),
\end{eqnarray*}
Denote by $\mathcal{F}(A)$ the set of all flag datums of $A$.
\end{defn}

\begin{prop}
Let $A$ be a pre-alternative algebra and $V$ a vector space of dimension $1$ with
basis $\{u\}$. Then there exists a bijection between the set
$Exd(A, V)$ of all  extending structures of $A$ through $V$ and the set ${\mathcal F} \, (A)$ of
all flag datums of $A$.

Through the above bijection, the unified product corresponding to
$(\lambda_\prec,\lambda_\succ$, $\mu_\prec$, $\mu_\prec$, $D_<$, $D_>$, $T_<$, $T_>$, $x_0, y_0, k_0, l_0)\in {\mathcal
F}(A)$ will be denoted by $A\natural \{u\}$ and has the multiplication given for
any $a$, $b \in A$ by:
\[
(x + u) \ll (y + u) = (x \prec y + D_ < (x) + T_ < (y) + y_0 ) + (l_0 u +
\lambda_\prec (x)u + \mu_\prec (y)u),
\]
\[
(x + u) \gg (y + u) = (x \succ y + D_ > (x) + T_ > (y) + x_0 ) + (k_0 u +
\lambda_\succ (x)u + \mu_\succ (y)u).
\]
That is $A\natural \{u\}$ is the pre-alternative generated by $A$ and $\{u\}$ subject
to the relations:
\[
u \ll u = y_0 + l_0 u,
\quad
u \gg u = x_0 + k_0 u,
\]
\[
x \ll u = D_ < (x) + \lambda_\prec (x)u,
\quad
x \gg u = D_ > (x) + \lambda_\succ (x)u,
\]
\[
u \ll x = T_ < (x) + \mu_\prec (x)u,
\quad
u \gg x = T_ > (x) + \mu_\succ (x)u.
\]
\end{prop}

\section{Appendix: Matched pair and unified product condition for pre-alternative algebras}
\begin{thm}
Let $A$ and $B$ be two pre-alternative algebras.
Then $(A,B)$  is  a matched pair if and only if the following conditions hold
$\forall x,y\in A$ and $u,v\in V$:
\begin{eqnarray}
\notag&& (x\diamond u) \succ y + (u\diamond x) \succ y + (u \circ x) > y +(x \circ u) > y \\
 &=& x > (u \succ y) + u > (x \succ y) + x \succ (u > y),
\end{eqnarray}
\begin{eqnarray}
(u \circ x) \succ y + (x \circ u) \succ y = x \succ (u \succ y) + u \succ (x
\succ y),
\end{eqnarray}
\begin{eqnarray}
(x\diamond u) > v + (u\diamond x) > v = x > (u > v) + u > (x > v),
\end{eqnarray}
\begin{eqnarray}
\notag&& (x\diamond u) \succ v + (u\diamond x) \succ v + (x \circ u) > v +(u \circ x) > v \\
 &=& x \succ (u > v) + u > (x \succ v) + u \succ (x > v) ,
\end{eqnarray}
\begin{eqnarray}
\notag&& (x \circ y) > u + (y \circ x) > u \\
 &=&x > (y \succ u) + y > (x \succ u) + y\succ (x > u) + x \succ (y > u),
\end{eqnarray}
\begin{eqnarray}
(x \circ y) \succ u + (y \circ x) \succ u = x \succ (y \succ u) + y \succ (x\succ u),
\end{eqnarray}
\begin{eqnarray}
(v\diamond u) > x + (u\diamond v) > x = u > (v > x) + v > (u > x),
\end{eqnarray}
\begin{eqnarray}
\notag&&(u\diamond v) \succ x + (v\diamond u) \succ x\\
 &=& v \succ (u > x) + u\succ (v > x) + u > (v \succ x) + v > (u \succ x),
\end{eqnarray}
\begin{eqnarray}
\notag&& x \prec (y\diamond u) + x \prec (u\diamond y) + x < (y \circ u) + x< (u \circ y) \\
 &=&  (x \prec u) < y + (x \prec y) < u + (x < u) \prec y,
\end{eqnarray}
\begin{eqnarray}
x \prec (u \circ y) + x \prec (y \circ u) = (x \prec u) \prec y + (x \prec y) \prec u,
\end{eqnarray}
\begin{eqnarray}
\notag&& u < (y \circ x) + u < (x \circ y) \\
&&\qquad=  (u < y) \prec x + (u \prec y) < x + (u < x) \prec y + (u \prec x) < y,
\end{eqnarray}
\begin{eqnarray}
u \prec (y \circ x) + u \prec (x \circ y) = (u \prec y) \prec x + (u \prec x) \prec y,
\end{eqnarray}
\begin{eqnarray}
x < (v\diamond u) + x < (u\diamond v) = (x < v) < u + (x < u) < v,
\end{eqnarray}
\begin{eqnarray}
\notag&& x \prec (u\diamond v) + x \prec (v\diamond u)\\
&&\qquad=  (x < u) \prec v + (x< v) \prec u + (x \prec u) < v + (x \prec v) < u,
\end{eqnarray}
\begin{eqnarray}
u < (x\diamond v) + u < (v\diamond x) = (u < x) < v + (u < v) < x,
\end{eqnarray}
\begin{eqnarray}
 \notag&& u < (x \circ v) + u < (v \circ x) + u \prec (v\diamond x) + u \prec (x\diamond v) \\
&&\qquad=  (u < x) \prec v + (u < v) \prec x + (u \prec x) < v,
\end{eqnarray}
\begin{eqnarray}
 \notag&&(x > u) \prec y + (u < x) \prec y + (x \succ u) < y + (u \prec x) < y \\
&&\qquad=  x \succ (u < y) + x > (u \prec y) + u < (x \circ y),
\end{eqnarray}
\begin{eqnarray}
(x \succ u) \prec y + (u \prec x) \prec y = u \prec (x \circ y) + x \succ (u\prec y),
\end{eqnarray}
\begin{eqnarray}
 \notag&&(u > x) \prec y + (u \succ x) < y + (x \prec u) < y + (x < u) \prec y \\
&&\qquad=  x \prec (u\diamond y) + u > (x \prec y) + x < (u \circ y),
\end{eqnarray}
\begin{eqnarray}
 (u \succ x) \prec y + (x \prec u) \prec y = x \prec (u \circ y) + u \succ (x\prec y),
\end{eqnarray}
\begin{eqnarray}
\notag&&  x > (y \prec u) + y \prec (x\diamond u) + y < (x \circ u) + x \succ (y <u) \\
 &=&  (x \succ y) < u + (y \prec x) < u,
\end{eqnarray}
\begin{eqnarray}
(x \succ y) \prec u + (y \prec x) \prec u = x \succ (y \prec u) + y \prec (x\circ u),
\end{eqnarray}
\begin{eqnarray}
(x > u) < v + (u < x) < v = x > (u < v) + u < (x\diamond v),
\end{eqnarray}
\begin{eqnarray}
 \notag&& (x > u) \prec v + (u < x) \prec v + (x \succ u) < v + (u \prec x) < v \\
&&\qquad=  x \succ (u < v) + u \prec (x\diamond v) + u < (x \circ v),
\end{eqnarray}
\begin{eqnarray}
(u > x) < v + (x < u) < v = u > (x < v) + x < (u\diamond v),
\end{eqnarray}
\begin{eqnarray}
\notag&& (u > x) \prec v + (x < u) \prec v\mbox{ + }(u \succ x) < v + (x \prec u) <v \\
&&\qquad=  x \prec (u\diamond v) + u > (x \prec v) + u \succ (x < v),
\end{eqnarray}
\begin{eqnarray}
(u < v) < x + (v < u) < x = u > (v < x) + v < (u\diamond x),
\end{eqnarray}
\begin{eqnarray}
\notag&&(u > v) \prec x + (v < u) \prec x \\
&&\qquad=  u \succ (v < x) + v \prec (u\diamond x) + u > (v \prec x) + v < (u \circ x),
\end{eqnarray}
\begin{eqnarray}
\notag&& (x > u) \prec y + (x \circ y) > u + (x \succ u) < y \\
&&\qquad=  x \succ (u < y) + x > (u \prec y) + x > (y \succ u) + x \succ (y > u),
\end{eqnarray}
\begin{eqnarray}
(x \succ u) \prec y + (x \circ y) \succ u = x \succ (u \prec y) + x \succ (y\succ u),
\end{eqnarray}
\begin{eqnarray}
\notag&& (u > x) \prec y + (u \succ x) < y + (u\diamond y) \succ x + (u \circ y)> x \\
&&\qquad=  u > (x \prec y) + u > (y \succ x),
\end{eqnarray}
\begin{eqnarray}
(u \succ x) \prec y + (u \circ y) \succ x = u \succ (y \succ x) + u \succ (x\prec y),
\end{eqnarray}
\begin{eqnarray}
 \notag&& (x \succ y) < u + (x \circ u) > y + (x\diamond u) \succ y \\
&&\qquad=  x > (u \succ y) + x \succ (u > y) + x \succ (y < u) + x > (y \prec u),
\end{eqnarray}
\begin{eqnarray}
(x \succ y) \prec u + (x \circ u) \succ y = x \succ (u \succ y) + x \succ (y \prec u),
\end{eqnarray}
\begin{eqnarray}
(x > u) < v + (x\diamond v) > u = x > (u < v) + x > (v > u),
\end{eqnarray}
\begin{eqnarray}
\notag&& (x \succ u) < v + (x \circ v) > u + (x > u) \prec v + (x\diamond v)\succ u \\
&&\qquad=  x \succ (v > u) + x \succ (u < v),
\end{eqnarray}
\begin{eqnarray}
(u > x) < v + (u\diamond v) > x = u > (x < v) + u > (v > x),
\end{eqnarray}
\begin{eqnarray}
\notag&& (u\diamond v) \succ x + (u \succ x) < v + (u > x) \prec v \\
&&\qquad=  u \succ (x < v) + u \succ (v > x) + u > (v \succ x) + u > (x \prec v),
\end{eqnarray}
\begin{eqnarray}
(u > v) < x + (u\diamond x) > v = u > (v < x) + u > (x > v),
\end{eqnarray}
\begin{eqnarray}
\notag&&(u > v) \prec x + (u\diamond x) \succ v \\
&&\qquad= u \succ (x > v) + u > (x \succ v) + u > (v \prec x) + u \succ (v < x).
\end{eqnarray}
\end{thm}

\begin{thm}
Let $A$ be a pre-alternative algebra, $V$ be a vector space and $\Omega(A,V)$ an extending datum of $A$ by $V$.
Denote by
\[
x \circ y = x \prec y + x \succ y,
\quad
x \circ v = x \prec v + x \succ v,
\quad
u \circ y = u \prec y + u \succ y,
\]
\[
u\diamond v = u < v + u > v,
\quad
x\diamond v = x < v + x > v,
\quad
u\diamond y = u < y + u > y,
\]
\[
\omega _\diamond (x,y) = \omega_< (u,v) + \omega_> (u,v).
\]
Then $A\natural V$ is a unified product if and only if the following conditions hold :
\begin{eqnarray}
\notag&& \omega _\diamond (u,v) > w + \omega _\diamond (v,u) > w + \omega _> (v\diamond u,w) + \omega_> (u\diamond v,w) \\
 &=& u > \omega_> (v,w) + v > \omega_> (u,w) + \omega_> (u,v > w) +\omega_> (v,u > w) ,
\end{eqnarray}
\begin{eqnarray}
\notag&& (u\diamond v) > w + (v\diamond u) > w + \omega _\diamond (u,v)\succ w + \omega _\diamond (v,u) \succ w \\
 &=& u \succ \omega_> (v,w) + v \succ \omega_> (u,w) + u > (v > w) + v >(u > w),
\end{eqnarray}
\begin{eqnarray}
\notag&&
 (x\diamond u) \succ y + (u\diamond x) \succ y + (u \circ x) > y +
(x \circ u) > y \\
 &=& x > (u \succ y) + u > (x \succ y) + x \succ (u > y),
\end{eqnarray}
\begin{eqnarray}
(u \circ x) \succ y + (x \circ u) \succ y = x \succ (u \succ y) + u \succ (x\succ y),
\end{eqnarray}
\begin{eqnarray}
\notag&&
 (u\diamond x) > v + (x\diamond u) > v + \omega_> (u \circ x,v) +\omega_> (x \circ u,v) \\
 &=& x > (u > v) + u > (x > v) + \omega_> (u,x \succ v) + x \succ \omega_>(u,v),
\end{eqnarray}
\begin{eqnarray}
\notag&&
 (x\diamond u) \succ v + (u\diamond x) \succ v + (x \circ u) > v +(u \circ x) > v \\
 &=& x \succ (u > v) + u \succ (x > v) + u > (x \succ v),
\end{eqnarray}
\begin{eqnarray}
\notag&& (x \circ y) > u + (y \circ x) > u \\
 &=& x > (y \succ u) + y \succ (x > u) + x\succ (y > u) + y > (x \succ u),
\end{eqnarray}
\begin{eqnarray}
(x \circ y) \succ u + (y \circ x) \succ u = x \succ (y \succ u) + y \succ (x
\succ u),
\end{eqnarray}
\begin{eqnarray}
\notag&& (v\diamond u) > x + (u\diamond v) > x + \omega _\diamond (u,v)\succ x + \omega _\diamond (v,u) \succ x \\
 &=& u > (v > x) + v > (u > x) + \omega_> (u,v > x) + \omega_> (v,u \succ x),
\end{eqnarray}
\begin{eqnarray}
\notag&&
(u\diamond v) \succ x + (v\diamond u) \succ x \\
&&\qquad= v \succ (u > x) + u
\succ (v > x) + u > (v \succ x) + v > (u \succ x),
\end{eqnarray}
\begin{eqnarray}
\notag&& \omega_< (u < v,w) + \omega_< (u < w,v) + \omega_< (u,w) < v + \omega_ < (u,v) < w \\
 &=& \omega_< (u,w\diamond v) + u < \omega _\diamond (w,v) + \omega_ < (u,v\diamond w) + u < \omega _\diamond (v,w),
\end{eqnarray}
\begin{eqnarray}
\notag&&
 (u < v) < w + (u < w) < v + \omega_< (u,w) \prec v + \omega_< (u,v)\prec w \\
 &=& u < (v\diamond w) + u < (w\diamond v) + u \prec \omega_\diamond (v,w) + u \prec \omega _\diamond (w,v),
\end{eqnarray}
\begin{eqnarray}
\notag&& (x \prec u) < y + (x \prec y) < u + (x < u) \prec y \\
 &=& x \prec (y\diamond u) + x \prec (u\diamond y) + x < (y \circ u) +x < (u \circ y),
\end{eqnarray}
\begin{eqnarray}
(x \prec u) \prec y + (x \prec y) \prec u = x \prec (u \circ y) + x \prec (y
\circ u),
\end{eqnarray}
\begin{eqnarray}
\notag&& (u < x) \prec y + (u \prec x) < y + (u < y) \prec x + (u \prec y) < x \\
&&\qquad= u <
(x \circ y) + u < (y \circ x),
\end{eqnarray}
\begin{eqnarray}
(u \prec x) \prec y + (u \prec y) \prec x = u \prec (x \circ y) + u \prec (y
\circ x),
\end{eqnarray}
\begin{eqnarray}
\notag&&
 (x < u) < v + (x < v) < u + \omega_< (x \prec v,u)\mbox{ + }\omega_< (x\prec u,v) \\
 &=& x < (v\diamond u) + x < (u\diamond v) + x \prec \omega_\diamond (v,u) + x \prec \omega _\diamond (u,v),
\end{eqnarray}
\begin{eqnarray}
\notag&&
(x < u) \prec v + (x < v) \prec u + (x \prec v) < u\mbox{ + }(x \prec u) < v\\
&&\qquad= x \prec (u\diamond v) + x \prec (v\diamond u),
\end{eqnarray}
\begin{eqnarray}
\notag&& (u < v) < x + (u < x) < v + \omega_< (u \prec x,v) + \omega_< (u,v)\prec x \\
 &=& u < (x\diamond v) + u < (v\diamond x) + \omega_< (u,x \circ v)+ \omega_< (u,v \circ x),
\end{eqnarray}
\begin{eqnarray}
\notag&& (u < x) \prec v + (u < v) \prec x + (u \prec x) < v \\
 &=& u \prec (x\diamond v) + u \prec (v\diamond x) + u < (v \circ x) +u < (x \circ v),
\end{eqnarray}
\begin{eqnarray}
\notag&&
 \omega_> (u,v) < w + \omega_< (u > v,w) + \omega_< (v < u,w) + \omega_ < (v,u) < w \\
 &=& u > \omega_< (v,w) + v < \omega _\diamond (u,w) + \omega_> (u,v <
w) + \omega_< (v,u\diamond w),
\end{eqnarray}
\begin{eqnarray}
\notag&& (u > v) < w + (v < u) < w + \omega_> (u,v) \prec w + \omega_< (v,u)\prec w \\
 &=& v < (u\diamond w) + u > (v < w) + u \succ \omega_< (v,w) + v \prec\omega _\diamond (u,w),
\end{eqnarray}
\begin{eqnarray}
\notag&&
 (x > u) \prec y + (u < x) \prec y + (x \succ u) < y + (u \prec x) < y \\
 &=& x \succ (u < y) + x > (u \prec y) + u < (x \circ y),
\end{eqnarray}
\begin{eqnarray}
(x \succ u) \prec y + (u \prec x) \prec y = u \prec (x \circ y) + x \succ (u\prec y),
\end{eqnarray}
\begin{eqnarray}
\notag&&
 (u > x) \prec y + (u \succ x) < y + (x \prec u) < y + (x < u) \prec y \\
 &=& x \prec (u\diamond y) + u > (x \prec y) + x < (u \circ y) ,
\end{eqnarray}
\begin{eqnarray}
(u \succ x) \prec y + (x \prec u) \prec y = x \prec (u \circ y) + u \succ (x\prec y),
\end{eqnarray}
\begin{eqnarray}
\notag&&
x > (y \prec u) + y \prec (x\diamond u) + y < (x \circ u)+ x\succ (y < u) \\
&&\qquad= (y \prec x) < u + (x \succ y) < u,
\end{eqnarray}
\begin{eqnarray}
(x \succ y) \prec u + (y \prec x) \prec u = y \prec (x \circ u) + x \succ (y\prec u),
\end{eqnarray}
\begin{eqnarray}
\notag&& (x > u) < v + (u < x) < v + \omega_< (x \succ u,v) + \omega_< (u \prec x,v) \\
 &=& x > (u < v) + u < (x\diamond v) + x \succ \omega_< (u,v) + \omega _< (u,x \circ v),
\end{eqnarray}
\begin{eqnarray}
\notag&&
 (u < x) \prec v + (x > u) \prec v + (x \succ u) < v + (u \prec x) < v \\
 &=& x \succ (u < v) + u \prec (x\diamond v) + u < (x \circ v),
\end{eqnarray}
\begin{eqnarray}
\notag&& (u > x) < v + (x < u) < v + \omega_< (u \succ x,v) + \omega_< (x \prec u,v) \\
 &=& u > (x < v) + x < (u\diamond v) + \omega_> (u,x \prec v) + x \prec\omega _\diamond (u,v),
\end{eqnarray}
\begin{eqnarray}
\notag&& (u > x) \prec v + (x < u) \prec v\mbox{ + }(u \succ x) < v + (x \prec u) <v \\
 &=& x \prec (u\diamond v) + u > (x \prec v) + u \succ (x < v),
\end{eqnarray}
\begin{eqnarray}
\notag&& (u < v) < x + (v < u) < x + \omega_> (u,v) \prec x + \omega_< (v,u)\prec x \\
 &=& u > (v < x) + v < (u\diamond x) + \omega_> (u,v \prec x) + \omega _< (v,u \circ x),
\end{eqnarray}
\begin{eqnarray}
\notag&&(u > v) \prec x + (v < u) \prec x \\
&&\qquad= u \succ (v < x) + v \prec (u\diamond x) + u > (v \prec x) + v < (u \circ x),
\end{eqnarray}
\begin{eqnarray}
\notag&& \omega _\diamond (u,w) > v + \omega_> (u\diamond w,v) + \omega _> (u,v) < w + \omega_< (u > v,w) \\
 &=& \omega_> (u,v < w) + u > \omega_> (w,v) + \omega_> (u,w > v) + u >\omega_< (v,w),
\end{eqnarray}
\begin{eqnarray}
\notag&& (u > v) < w + (u\diamond w) > v + \omega_> (u,v) \prec w + \omega_\diamond (u,w) \succ v \\
 &=& u > (v < w) + u > (w > v) + u \succ \omega_< (v,w) + u \succ \omega_>(w,v),
\end{eqnarray}
\begin{eqnarray}
\notag&& (x > u) \prec y + (x \circ y) > u + (x \succ u) < y \\
 &=& x \succ (u < y) + x > (u \prec y) + x > (y \succ u) + x \succ (y > u),
\end{eqnarray}
\begin{eqnarray}
(x \succ u) \prec y + (x \circ y) \succ u = x \succ (u \prec y) + x \succ (y\succ u),
\end{eqnarray}
\begin{eqnarray}
\notag&&(u > x) \prec y + (u \succ x) < y + (u\diamond y) \succ x + (u \circ y)> x \\
&&\qquad= u > (x \prec y) + u > (y \succ x),,
\end{eqnarray}
\begin{eqnarray}
(u \succ x) \prec y + (u \circ y) \succ x = u \succ (y \succ x) + u \succ (x
\prec y),
\end{eqnarray}
\begin{eqnarray}
\notag&&
 (x \succ y) < u + (x \circ u) > y + (x\diamond u) \succ y \\
 &=& x > (u \succ y) + x \succ (u > y) + x \succ (y < u) + x > (y \prec u),
\end{eqnarray}
\begin{eqnarray}
(x \succ y) \prec u + (x \circ u) \succ y = x \succ (u \succ y) + x \succ (y\prec u),
\end{eqnarray}
\begin{eqnarray}
\notag&& (x > u) < v + (x\diamond v) > u + \omega_< (x \succ u,v) + \omega_>(x \circ v,u) \\
 &=& x > (u < v) + x > (v > u) + x \succ \omega_< (u,v) + x \succ \omega_>(v,u),
\end{eqnarray}
\begin{eqnarray}
\notag&&
(x \succ u) < v + (x \circ v) > u + (x > u) \prec v + (x\diamond v)\succ u \\
&&\qquad= x \succ (v > u) + x \succ (u < v),
\end{eqnarray}
\begin{eqnarray}
\notag&& (u > x) < v + (u\diamond v) > x + \omega_< (u \succ x,v) + \omega_\diamond (u,v) \succ x \\
 &=& u > (x < v) + u > (v > x) + \omega_> (u,x \prec v) + \omega_> (u,v\succ x),
\end{eqnarray}
\begin{eqnarray}
\notag&& (u\diamond v) \succ x + (u \succ x) < v + (u > x) \prec v \\
 &=& u \succ (x < v) + u \succ (v > x) + u > (v \succ x) + u > (x \prec v) ,
\end{eqnarray}
\begin{eqnarray}
\notag&& (u > v) < x + (u\diamond x) > v + \omega_> (u,v) \prec x + \omega_>(u \circ x,v) \\
 &=& u > (v < x) + u > (x > v) + \omega_> (u,v \prec x) + \omega_> (u,x\succ v),
\end{eqnarray}
\begin{eqnarray}
\notag&&(u > v) \prec x + (u\diamond x) \succ v \\
&&\qquad= u \succ (x > v) + u > (x \succ v) + u > (v \prec x) + u \succ (v < x).
\end{eqnarray}
\end{thm}

\subsection*{Acknowledgements}
This is a primary edition, some example will be added in the future.

\vskip7pt
\footnotesize{
\noindent Tao Zhang\\
College of Mathematics and Information Science,\\
Henan Normal University, Xinxiang 453007, P. R. China;\\
 E-mail address:\texttt{{zhangtao@htu.edu.cn}}

\vskip7pt
\footnotesize{
\noindent Shuxian Cui\\
Center for Applied Mathematics at Tianjin University,\\
Tianjin University, Tianjin 300072, P.R.China;\\
 E-mail address:\texttt{{csxcuichen@163.com}}
}

\vskip7pt
\footnotesize{
\noindent Jing Si\\
College of Mathematics and Information Science,\\
Henan Normal University, Xinxiang 453007, P. R. China;\\
 E-mail address:\texttt{{hnszksj@163.com}}
}

}

\end{document}